\documentclass[12pt]{article}

\usepackage{amsthm}
\usepackage{amscd}
\usepackage{amsfonts}
\usepackage{amssymb}
\usepackage{amsgen}
\usepackage{amsmath}
\usepackage{amsopn}
\usepackage{verbatim}
\usepackage{xypic}
\usepackage{xspace}
\usepackage{multicol}
\usepackage{url}
\usepackage{upref}

\theoremstyle{plain}
\newtheorem{thm}{Theorem}[section]
\newtheorem{lem}[thm]{Lemma}

\theoremstyle{definition}

\newtheorem{rem}[thm]{Remark}

 \font\cyr=wncyr10
 \newcommand{\nc}{\newcommand}

\nc{\per}[1]{\underset{#1}{\boldsymbol \pi}\,}

 \nc{\MT}{{\rm MT}}
 \nc{\XX}{{X}}
 \nc{\gF}{{\varPhi}}
 \nc{\ot}{\otimes}
\nc{\bfrac}{\displaystyle\bfrac}
 \nc{\evalE}{{\mathfrak Z}}
 \nc{\tG}{\tilde{G}}
 \nc{\wht}{\widehat}
 \nc{\bwg}{{\bigwedge}}
 \nc{\wg}{{\wedge}}
 \nc{\mmu}{{\boldsymbol{\mu}}}
 \nc{\mal}{{{\scriptstyle \maltese}}}
 \nc{\fA}{{\mathfrak A}}
 \nc{\HH}{{\mathfrak H}}
 \nc{\ra}{\rightarrow}
 \nc{\ors}{{\bfs}}
 \nc{\orr}{{\bfr}}
 \nc{\os}{{\overset}}
 \nc{\G}{{\mathbb G}}
 \nc{\F}{{\mathbb F}}
 \nc{\Z}{{\mathbb Z}}
 \nc{\R}{{\mathbb R}}
 \nc{\N}{{\mathbb N}}
 \nc{\ZN}{{\mathbb Z_{\ge 0}}}
 \nc{\Q}{{\mathbb Q}}
 \nc{\C}{{\mathbb C}}
 \nc{\CP}{{\mathbb{CP}}}
 \nc{\Cnn}{{\mathbb C}_{\ge 0}}
 \nc{\Cp}{{\mathbb C}_{>0}}
 \nc{\MPV}{{\mathcal{MPV}}}
 
 \nc{\tB}{{\tilde B}}
 \nc{\tP}{{P}}
 \nc{\tx}{{\tilde x}}
 \nc{\ty}{{\tilde y}}
 \nc{\tz}{{\tilde z}}
 \nc{\tw}{{\tilde w}}
 \nc{\tQ}{{Q}}
\nc{\gemn}{{\mathfrak n}}
 \nc{\suf}{{\ast\,}}
 \nc{\sufq}{{\ast_q\,}}
 \nc{\gam}{{\gamma}}
 \nc{\gG}{{\Gamma}}
 \nc{\om}{{\omega}}
 \nc{\vep}{{\varepsilon}}
 \nc{\ga}{{\alpha}}
 \nc{\gl}{{\lambda}}
 \nc{\gb}{{\beta}}
 \nc{\gf}{{\varphi}}
 \nc{\gd}{{\delta}}
 \nc{\orgd}{{\vec \gd\,}}
 \nc{\gs}{{\sigma}}
 \nc{\gth}{{\theta}}
 \nc{\gx}{{\xi}}
 \nc{\gS}{{\Sigma}}

 \nc{\gk}{{\kappa}}
  \nc{\gz}{{\zeta}}
 \nc{\tgz}{{\tilde{\zeta}}}
 \nc{\gO}{{\Omega}}
 \nc{\sif}{{\mathcal S}}
 \nc{\gt}{{\tau}}
 \nc{\Lra}{\Longrightarrow}
 \nc{\lra}{\longrightarrow}
 \nc{\lmaps}{\longmapsto}
 \nc{\fS}{{\mathfrak S}}
 \nc{\DD}{{\mathfrak D}}
 \nc{\Llra}{\Longleftrightarrow}
 \nc{\ol}{\overline}
 \nc{\ola}{\overleftarrow}
 \nc{\lms}{\longmapsto}
 \nc{\cv}{{{\mathsf c}{\mathsf v}}}
 \nc{\zq}{{\zeta_q}}
 \nc\qup{{q\uparrow 1}}
 \nc{\us}{\underset}
 \nc{\tn}{{\tilde{n}}}
 \nc{\gD}{{\Delta}}
 \nc{\bi}{{\bf i}}
 \nc{\bfone}{{\bf 1}}

 \nc{\bfa}{{\bf a}}
 \nc{\bfb}{{\bf b}}
 \nc{\bfc}{{\bf c}}
 \nc{\bfd}{{\bf d}}
 \nc{\bfe}{{\bf e}}
 \nc{\bff}{{\bf f}}
 \nc{\bfg}{{\bf g}}
 \nc{\bfi}{{\bf i}}
 \nc{\bfj}{{\bf j}}
\nc{\obfi}{{\overrightarrow{\boldsymbol \imath}}}
\nc{\obfj}{{\overrightarrow{\boldsymbol \jmath}}}
\nc{\obfd}{{\overrightarrow{\bf d}}}
\nc{\veps}{{\varepsilon}}
 \nc{\bfn}{{\bf n}}
 \nc{\bfl}{{\bf l}}
 \nc{\bfk}{{\bf k}}
 \nc{\bfm}{{\bf m}}
 \nc{\bfo}{{\bf o}}
 \nc{\bfp}{{\bf p}}
 \nc{\bfq}{{\bf q}}
 \nc{\bfr}{{\bf r}}
 \nc{\bfs}{{\bf s}}
 \nc{\bft}{{\bf t}}
 \nc{\bfu}{{\bf u}}
 \nc{\bfv}{{\bf v}}
 \nc{\bfw}{{\bf w}}
 \nc{\bfx}{{\bf x}}
 \nc{\bfy}{{\bf y}}
 \nc{\bfz}{{\bf z}}
 \nc{\bfB}{{\bf B}}
 \nc{\bfP}{{\bf P}}
 \nc{\bfQ}{{\bf Q}}
 \nc{\bfY}{{\bf Y}}
 \nc{\bfgb}{{\boldsymbol \gb}}
 \nc{\bfga}{{\boldsymbol \ga}}
 \nc{\bfrho}{{\boldsymbol \rho}}
 \nc{\bfchi}{{\boldsymbol \chi}}
 \nc{\QX}{{\Q\langle \bfX\rangle}}
 \nc{\QY}{{\Q\langle \bfY\rangle}}
 \nc{\CX}{{\C\langle \bfX\rangle}}
 \nc{\CY}{{\C\langle \bfY\rangle}}
 \nc{\QXX}{{\Q\langle\!\langle \bfX\rangle\!\rangle}}
 \nc{\QYY}{{\Q\langle\!\langle \bfY\rangle\!\rangle}}
 \nc{\CXX}{{\C\langle\!\langle \bfX\rangle\!\rangle}}
 \nc{\CYY}{{\C\langle\!\langle \bfY\rangle\!\rangle}}

 \nc{\bbA}{{\mathbb A}}
 \nc{\bbB}{{\mathbb B}}
 \nc{\bbC}{{\mathbb C}}
 \nc{\bbD}{{\mathbb D}}
 \nc{\bbE}{{\mathbb E}}
 \nc{\bbF}{{\mathbb F}}
 \nc{\bbG}{{\mathbb G}}
 \nc{\bbH}{{\mathbb H}}
 \nc{\bbI}{{\mathbb I}}
 \nc{\bbJ}{{\mathbb J}}
 \nc{\bbK}{{\mathbb K}}
 \nc{\bbL}{{\mathbb L}}
 \nc{\bbM}{{\mathbb M}}
 \nc{\bbN}{{\mathbb N}}
 \nc{\bbO}{{\mathbb O}}
 \nc{\bbP}{{\mathbb P}}
 \nc{\bbQ}{{\mathbb Q}}
 \nc{\bbR}{{\mathbb R}}
 \nc{\bbS}{{\mathbb S}}
 \nc{\bbT}{{\mathbb T}}
 \nc{\bbU}{{\mathbb U}}
 \nc{\bbV}{{\mathbb V}}
 \nc{\bbW}{{\mathbb W}}
 \nc{\bbX}{{\mathbb X}}
 \nc{\bbY}{{\mathbb Y}}
 \nc{\bbZ}{{\mathbb Z}}
 \nc{\bba}{{\mathbb a}}
 \nc{\bbb}{{\mathbb b}}
 \nc{\bbc}{{\mathbb c}}
 \nc{\bbd}{{\mathbb d}}
 \nc{\bbe}{{\mathbb e}}
 \nc{\bbf}{{\mathbb f}}
 \nc{\bbg}{{\mathbb g}}
 \nc{\bbh}{{\mathbb h}}
 \nc{\bbi}{{\mathbb i}}
 \nc{\bbk}{{\mathbb k}}
 \nc{\bbl}{{\mathbb l}}
 \nc{\bbm}{{\mathbb m}}
 \nc{\bbn}{{\mathbb n}}
 \nc{\bbo}{{\mathbb o}}
 \nc{\bbp}{{\mathbb p}}
 \nc{\bbq}{{\mathbb q}}
 \nc{\bbr}{{\mathbb r}}
 \nc{\bbs}{{\mathbb s}}
 \nc{\bbt}{{\mathbb t}}
 \nc{\bbu}{{\mathbb u}}
 \nc{\bbv}{{\mathbb v}}
 \nc{\bbw}{{\mathbb w}}
 \nc{\bbx}{{\mathbb x}}
 \nc{\bby}{{\mathbb y}}
 \nc{\bbz}{{\mathbb z}}

 \nc{\MZV}{{\mathcal{MZV}}}
 \nc{\calA}{{\mathcal A}}
 \nc{\calB}{{\mathcal B}}
 \nc{\calC}{{\mathcal C}}
 \nc{\calD}{{\mathcal D}}
 \nc{\calE}{{\mathcal E}}
 \nc{\calF}{{\mathcal F}}
 \nc{\calG}{{\mathcal G}}
 \nc{\calH}{{\mathcal H}}
 \nc{\calI}{{\mathcal I}}
 \nc{\calJ}{{\mathcal J}}
 \nc{\calK}{{\mathcal K}}
 \nc{\calL}{{\mathcal L}}
 \nc{\calM}{{\mathcal M}}
 \nc{\calN}{{\mathcal N}}
 \nc{\calO}{{\mathcal O}}
 \nc{\calP}{{\mathcal P}}
 \nc{\calQ}{{\mathcal Q}}
 \nc{\calR}{{\mathcal R}}
 \nc{\calS}{{\mathcal S}}
 \nc{\calT}{{\mathcal T}}
 \nc{\calU}{{\mathcal U}}
 \nc{\calV}{{\mathcal V}}
 \nc{\calW}{{\mathcal W}}
 \nc{\calX}{{\mathcal X}}
 \nc{\calY}{{\mathcal Y}}
 \nc{\calZ}{{\mathcal Z}}
  \nc{\cala}{{\mathcal a}}
 \nc{\calb}{{\mathcal b}}
 \nc{\calc}{{\mathcal c}}
 \nc{\cald}{{\mathcal d}}
 \nc{\cale}{{\mathcal e}}
 \nc{\calf}{{\mathcal f}}
 \nc{\calg}{{\mathcal g}}
 \nc{\calh}{{\mathcal h}}
 \nc{\cali}{{\mathcal i}}
 \nc{\calj}{{\mathcal j}}
 \nc{\calk}{{\mathcal k}}
 \nc{\call}{{\mathcal l}}
 \nc{\calm}{{\mathcal m}}
 \nc{\caln}{{\mathcal n}}
 \nc{\calo}{{\mathcal o}}
 \nc{\calp}{{\mathsf p}}
 \nc{\calq}{{\mathcal q}}
 \nc{\calr}{{\mathcal r}}
 \nc{\cals}{{\mathcal s}}
 \nc{\calt}{{\mathcal t}}
 \nc{\calu}{{\mathcal u}}
 \nc{\calv}{{\mathcal v}}
 \nc{\calw}{{\mathcal w}}
 \nc{\calx}{{\mathcal x}}
 \nc{\caly}{{\mathcal y}}
 \nc{\calz}{{\mathcal z}}

 \nc{\frakA}{{\mathfrak A}}
 \nc{\frakB}{{\mathfrak B}}
 \nc{\frakC}{{\mathfrak C}}
 \nc{\frakD}{{\mathfrak D}}
 \nc{\frakE}{{\mathfrak E}}
 \nc{\frakF}{{\mathfrak F}}
 \nc{\frakG}{{\mathfrak G}}
 \nc{\frakH}{{\mathfrak H}}
 \nc{\frakI}{{\mathfrak I}}
 \nc{\frakJ}{{\mathfrak J}}
 \nc{\frakK}{{\mathfrak K}}
 \nc{\frakL}{{\mathfrak L}}
 \nc{\frakM}{{\mathfrak M}}
 \nc{\frakN}{{\mathfrak N}}
 \nc{\frakO}{{\mathfrak O}}
 \nc{\frakP}{{\mathfrak P}}
 \nc{\frakQ}{{\mathfrak Q}}
 \nc{\frakR}{{\mathfrak R}}
 \nc{\frakS}{{\mathfrak S}}
 \nc{\frakT}{{\mathfrak T}}
 \nc{\frakU}{{\mathfrak U}}
 \nc{\frakV}{{\mathfrak V}}
 \nc{\frakW}{{\mathfrak W}}
 \nc{\frakX}{{\mathfrak X}}
 \nc{\frakY}{{\mathfrak Y}}
 \nc{\frakZ}{{\mathfrak Z}}
 \nc{\fraka}{{\mathfrak a}}
 \nc{\frakb}{{\mathfrak b}}
 \nc{\frakc}{{\mathfrak c}}
 \nc{\frakd}{{\mathfrak d}}
 \nc{\frake}{{\mathfrak e}}
 \nc{\frakf}{{\mathfrak f}}
 \nc{\frakg}{{\mathfrak g}}
 \nc{\frakh}{{\mathfrak h}}
 \nc{\fraki}{{\mathfrak i}}
 \nc{\frakj}{{\mathfrak j}}
 \nc{\frakk}{{\mathfrak k}}
 \nc{\frakl}{{\mathfrak l}}
 \nc{\frakm}{{\mathfrak m}}
 \nc{\frakn}{{\mathfrak n}}
 \nc{\frako}{{\mathfrak o}}
 \nc{\frakp}{{\mathfrak p}}
 \nc{\frakq}{{\mathfrak q}}
 \nc{\frakr}{{\mathfrak r}}
 \nc{\fraks}{{\mathfrak s}}
 \nc{\frakt}{{\mathfrak t}}
 \nc{\fraku}{{\mathfrak u}}
 \nc{\frakv}{{\mathfrak v}}
 \nc{\frakw}{{\mathfrak w}}
 \nc{\frakx}{{\mathfrak x}}
 \nc{\fraky}{{\mathfrak y}}
 \nc{\frakz}{{\mathfrak z}}
 \nc{\so}{{\mathfrak so}}
 \nc{\sa}{{\mbox{{\scriptsize \cyr x}}}}
 \nc{\slfour}{{\mathfrak sl}_4}
 \nc{\one}{{\bf 1}}
 \nc{\zero}{{\bf 0}}
 \nc{\Qxy}{\Q\langle x,y\rangle}

 \nc{\barN}{{\ol{\mathbb N}}}

\begin{document}
\date{}
\begin{center}
 {\bf\large Restricted Sum Formula of Multiple Zeta Values}
\end{center}

\begin{center}
 {Haiping Yuan\\
York College of Pennsylvania, York, PA 17403}
\end{center}

\begin{center}
 {Jianqiang Zhao\\
Department of Mathematics, Eckerd College, St. Petersburg, FL 33711}
\end{center}

\section{Introduction}

For fixed positive integer $d$ and $d$-tuple of positive integers $(s_1,\dots, s_d)$ with $s_1>1$,
the multiple zeta value $\zeta(s_1,\dots, s_d)$ is defined by
\begin{equation}\label{equ:MZV}
\zeta(s_1,\dots, s_d)=\sum_{k_1>\dots>k_d>0} k_1^{-s_1}\cdots k_d^{-s_d},
\end{equation}
where $d$ is called the \emph{depth} and
$s_1+\dots+s_d$ the \emph{weight}. The double zeta values were
studied by Euler \cite{Euler} who derived many identities such as follows:
\begin{align*}
\sum_{k=2}^{2n-1}(-1)^k\zeta(k,2n-k)&=\frac12\zeta(2n), \\
\sum_{k=2}^{2n-1}\zeta(k,2n-k)&=\zeta(2n),
\end{align*}
from which we can easily get (see \cite[Theorem 1]{GKZ})
\begin{equation}\label{equ:GKZDblZetaEven}
\sum_{k=1}^{n-1}\zeta(2k,2n-2k)=\frac34\zeta(2n).
\end{equation}
Using the stuffle relation $\zeta(2k)\zeta(2n-2k)=\zeta(2k,2n-2k)+\zeta(2n-2k,2k)+\zeta(2n)$
we see immediately
\begin{equation}\label{equ:dbld=2}
\sum_{k=1}^{n-1}\zeta(2k)\zeta(2n-2k)=\frac{2n+1}2 \zeta(2n).
\end{equation}

Recently, Hoffman \cite{Hoffman2012} extended \eqref{equ:GKZDblZetaEven} to arbitrary depths.
Moreover, similar formulas have been obtained for some special type Hurwitz-zeta values \cite{ZTn}
and alternating Euler sums \cite{ZEulerSum}.
In this paper we consider the following restricted sum of multiple zeta values
\begin{equation*}
Q(4n,d)=\sum_{\substack{j_1+\cdots+j_d=n\\ j_1,\dots,j_d > 0}}
\zeta(4j_1,\dots,4j_d).
\end{equation*}
Our main theorem is
\begin{thm}\label{thm:main}
For any positive integers $n\ge d\ge 3,$
\begin{multline*}
Q(4n,d)=\sum_{k=0}^{\lfloor\frac{d-1}{2}\rfloor}
\sum_{j=0}^{2k+1}\frac{2^{k+2}(-1)^{\lfloor\frac{k}{2}\rfloor +j+d}}{(2k+1)!} \binom{2k+1}{j}\binom{\frac{j-2}{4}}{d} \zeta(4n-2k)\pi^{2k} \\
+\sum_{k=0}^{\lfloor\frac{d-2}{4}\rfloor}\sum_{j=0}^{4k+2}
    \frac{2^{2k+5}(-1)^{k+j+d}}{(4k+2)!}\binom{4k+2}{j}\binom{\frac{j-2}{4}}{d}
    \left(Q(4n-4k,2)-\frac{7}{8}\zeta(4n-4k) \right) \pi^{4k}.
\end{multline*}
\end{thm}

\begin{rem} For $d=2,$ it's easy to prove by stuffle relation that
$$Q(4n, 2) = \frac{1}{2}\sum_{k=1}^{n-1}\zeta(4k)\zeta(4n-4k)-\frac{n-1}{2}\zeta(4n)$$
for $n\geq 2.$ However, it is an intriguing problem to find a compact formula similar to
\eqref{equ:dbld=2}.
\end{rem}

\noindent
{\bf Acknowledgement.} Both authors would like to thank the Morningside Center of Mathematics, Chineses Academy of Science for hospitality when the paper was prepared. HY is partially supported by the Summer Research Grant from York College of Pennsylvania and  JZ is partially supported by NSF grant DMS1162116. They also want to thank the anonymous referee who pointed out a few inaccuracies and possible improvement in the original draft.

\section{The generating function of $Q(4n,d)$}
Recall that the symmetric function of the infinitely many variables $x_1, x_2,\cdots$
form a subring Sym of $\bbQ[x_1, x_2,\cdots]$ which is invariant under all the permutations
of the variables. Let $e_j=\sum_{k_1<\dots<k_j} x_{k_1}\dots x_{k_j}$ be the $j$-th
elementary function. Following Hoffman \cite{Hoffman2012} let's consider its
generating function
\begin{equation*}
E(t)= \prod_{j=1}^{\infty} (1+tx_j) = \sum_{j=0}^{\infty} e_j t^j
\end{equation*}
and define $\vep: {\rm Sym}\to \R$ to be the evaluation map such that
$\vep(x_j)= \displaystyle\frac{1}{j^4}.$ Let
\begin{equation*}
    F(s,t)=\prod_{j=1}^{\infty}(1+ tsx_j+ ts^2x_j^2+ \cdots).
\end{equation*}
Then it is not hard to see that the generating function of $Q(4n,d)$ is given by
\begin{equation*}
\vep\big( F(s,t) \big) = \sum_{n=0}^{\infty} Q(4n,d) t^ds^n.
\end{equation*}
First we need the following lemma.
\begin{lem}
We have
\begin{equation*}
 \vep(F(s,t)) =\frac{ \sin \pi\sqrt[4]{s(1-t)}\cdot \sinh \pi\sqrt[4]{s(1-t)}}
 {\sqrt{1-t}  \sin \pi \sqrt[4]{s} \cdot \sinh \pi\sqrt[4]{s} }.
\end{equation*}
\end{lem}
\begin{proof}
We have
\begin{align*}
\prod_{j=1}^{\infty} (1+ tsx_j+ ts^2x_j^2+ \cdots)
=&\prod_{j=1}^{\infty} \left(1+ t\frac{sx_j}{1-sx_j}\right)\\
=&\frac{\prod_{j=1}^{\infty}(1-s(1-t)x_j)}{\prod_{j=1}^{\infty}(1-sx_j)}
= \frac{E(-s(1-t))}{E(-s)}.
\end{align*}
Further,
\begin{equation*}
\vep(E(-t)) = \prod_{i=1}^{\infty} \left(1-\frac{t}{i^4}\right)
= \prod_{i=1}^{\infty}\left(1- \frac{\sqrt{t}}{i^2}\right)\left(1+\frac{\sqrt{t}}{i^2}\right)
= \frac{\sin \pi\sqrt[4]{t} \cdot\sinh \pi\sqrt[4]{t}} {\pi^2\sqrt{t} }.
\end{equation*}
The lemma follows immediately.
\end{proof}

Let $f(x)=  \sin x \cdot  \sinh x/(2x^2).$
The following lemma provides its series expansion.
\begin{lem}
We have
\begin{equation*}
f(x)= \sum_{k=0}^{\infty} \frac{(-1)^k 4^k}{(4k+2)!} x^{4k}.
\end{equation*}
\end{lem}
\begin{proof}
Using the well-known formula $\sin x= (e^{ix}-e^{-ix})/(2i)$
we obtain
\begin{align*}
f(x)&= \frac12\cdot \frac{e^{ix}-e^{-ix}}{2ix} \cdot \frac{e^{x}-e^{-x}}{2x}\\
&= \frac{e^{(i+1)x}+e^{-(i+1)x}-(e^{(i-1)x}+e^{-(i-1)x})}{8ix^2}\\
&= \frac{1}{4ix^2}\left(\sum_{n=0}^{\infty}\frac{(2i)^{n}x^{2n}}{(2n)!}- \sum_{n=0}^{\infty}\frac{(-2i)^nx^{2n}}{(2n)!}\right)\\
&= \sum_{k=0}^{\infty}\frac{(-1)^k4^k}{(4k+2)!}x^{4k},
\end{align*}
as desired.
\end{proof}

\section{Proof of Theorem \ref{thm:main}}
Let $g(t)= f(\sqrt[4]{t}).$
Then
\begin{equation*}
\frac{g(s(1-t))}{g(s)}=\vep(F(s/\pi^4,t))
=\frac{1}{g(s)}\sum_{k=0}^{\infty} \frac{(-1)^k 4^k}{(4k+2)!} s^{k}(1-t)^k.
\end{equation*}
Write
\begin{equation*}
\frac{g(s(1-t))}{g(s)}=\sum_{d=0}^{\infty} G_d(s)t^d.
\end{equation*}
By the above expression, we have
\begin{equation*}
G_d(s)=\frac{(-s)^d }{g(s) d!}D^d g(s),
\end{equation*}
where $D^d$ denotes the $d$-th derivative with respect to $s$. Set
\begin{equation} \label{equ:GdsDefn}
G_d(s)=X_d(s)\sqrt[4]{s}\cot \sqrt[4]{s} +Y_d(s)\sqrt[4]{s}\coth \sqrt[4]{s}
+Z_d(s)\cot \sqrt[4]{s}\coth \sqrt[4]{s}+ W_d(s)
\end{equation}
which yields easily
\begin{align*}
\frac{(-1)^sD^d g(s)}{d!}
&= X_d(s)s^{-d-\frac{1}{4}}\cos s^{\frac{1}{4}}\sinh s^{\frac{1}{4}}
    + Y_d(s)s^{-d-\frac{1}{4}}\sin s^{\frac{1}{4}}\cosh s^{\frac{1}{4}}\\
&+Z_d(s)s^{-d-\frac{1}{2}}\cos s^{\frac{1}{4}}\cosh s^{\frac{1}{4}}
    + W_d(s)s^{-d-\frac{1}{2}}\sin s^{\frac{1}{4}}\sinh s^{\frac{1}{4}}.
\end{align*}
To determine the coefficients $X_d(s), Y_d(s), Z_d(s)$ and $W_d(s)$ we
differentiate the both sides of the above equation to get the following
system of recursive differential equations
\begin{equation*}
\left\{\aligned
(d+1)X_{d+1}(s) &= -sX'_d(s)+ \Big(d+\frac{1}{4}\Big)X_d(s) -\frac{1}{4}Z_d(s)
-\frac{1}{4}W_d(s),\\
(d+1)Y_{d+1}(s) &= -sY'_d(s)+ \Big(d+\frac{1}{4}\Big)Y_d(s) +
\frac{1}{4}Z_d(s)-\frac{1}{4}W_d(s),\\
(d+1)Z_{d+1}(s)&=
-\frac{\sqrt{s}}{4}X_d(s) -\frac{\sqrt{s}}{4}Y_d(s) -
s Z'_d(s) +\Big(d+\frac{1}{2}\Big)Z_d(s),\\
(d+1)W_{d+1}(s)&= \frac{\sqrt{s}}{4}X_d(s) -
\frac{\sqrt{s}}{4}Y_d(s) + \Big(d+\frac{1}{2}\Big)W_d(s) - sW'_d(s),
\endaligned\right.
\end{equation*}
with the initial conditions $X_0(s)=Y_0(s)=Z_0(s)=0$ and $W_0(s)=1$.
Let $x_d(u)= X_d(u^2), y_d(u)=Y_d(u^2), z_d(u)=Z_d(u^2)$ and $w_d(u)=
W_d(u^2).$ The above system is changed into the following system:
\begin{equation}\label{equ:DEsystem}
\left\{\aligned
(d+1)x_{d+1}(u)&= -\frac{u}{2}x'_d(u)+\Big(d+\frac{1}{4}\Big)x_d(u)-
\frac{1}{4}z_d(u)-\frac{1}{4}w_d(u),\\
(d+1)y_{d+1}(u)&=
-\frac{u}{2}y'_d(u)+\Big(d+\frac{1}{4}\Big)y_d(u)+\frac{1}{4}z_d(u)-\frac{1}{4}w_d(u),\\
(d+1)z_{d+1}(u)&= -\frac{u}{4}x_d(u)
-\frac{u}{4}y_d(u)-\frac{u}{2}z'_d(u)+\Big(d+\frac{1}{2}\Big)z_d(u),\\
(d+1)w_{d+1}(u)&=
\frac{u}{4}x_d(u)-\frac{u}{4}y_d(u)+\Big(d+\frac{1}{2}\Big)w_d(u)-\frac{u}{2}w'_d(u).
\endaligned\right.
\end{equation}
Define
\begin{equation}\label{equ:genFall}
\left\{\aligned
\ga(u,v)=&\sum_{d\geq 0} x_d(u)v^d=\sum_{d\ge 0} \tx_d(v)u^d,\\
\gb(u,v)=&\sum_{d\geq 0} y_d(u)v^d=\sum_{d\ge 0} \ty_d(v)u^d,\\
\gam(u,v)=&\sum_{d\geq 0} z_d(u)v^d=\sum_{d\ge 0} \tz_d(v)u^d,\\
\gd(u,v)=&\sum_{d\geq 0} w_d(u)v^d=\sum_{d\ge 0} \tw_d(v)u^d.
\endaligned\right.
\end{equation}
Multiplying the system \eqref{equ:DEsystem} by $v^d$ and then taking the sum
$\sum_{d\ge 0}$ we get:
\begin{equation*}
\left\{\aligned
    \frac{\partial \ga}{\partial v}=&v \frac{\partial \ga}{\partial
v}+\frac14 \ga
    -\frac{u}2 \frac{\partial \ga}{\partial u}-\frac14 \gam-\frac14 \gd,\\
    \frac{\partial \gb}{\partial v}=&v \frac{\partial \gb}{\partial
v}+\frac14 \gb
    -\frac{u}2 \frac{\partial \gb}{\partial u}+\frac14 \gam-\frac14 \gd,\\
    \frac{\partial \gam}{\partial v}=&v \frac{\partial \gam}{\partial
v}+\frac12 \gam
    -\frac{u}2 \frac{\partial \gam}{\partial u}-\frac{u}4 \ga-\frac{u}4 \gb,\\
    \frac{\partial \gd}{\partial v}=&v \frac{\partial \gd}{\partial
v}+\frac12 \gd
    -\frac{u}2 \frac{\partial \gd}{\partial u}+\frac{u}4 \ga-\frac{u}4 \gb.
\endaligned\right.
\end{equation*}
Comparing the coefficients of $u^n$ we get
\begin{equation}\label{equ:DEsystemTilde}
\left\{\aligned
\tx_n'(v)=&v\tx_n'(v)+\frac14\tx_n(v)-\frac{n}2\tx_n(v)-\frac14\tz_n(v)-\frac14\tw_n(v),\\
\ty_n'(v)=&v\ty_n'(v)+\frac14\ty_n(v)-\frac{n}2\ty_n(v)+\frac14\tz_n(v)-\frac14\tw_n(v),\\
\tz_n'(v)=&v\tz_n'(v)+\frac12\tz_n(v)-\frac{n}2\tz_n(v)-\frac14\tx_{n-1}(v)-\frac14\ty_{n-1}(v),\\
\tw_n'(v)=&v\tw_n'(v)+\frac12\tw_n(v)-\frac{n}2\tw_n(v)+\frac14\tx_{n-1}(v)-\frac14\ty_{n-1}(v),\\
\endaligned\right.
\end{equation}
By definition \eqref{equ:genFall}, we see that the system has the following initial values:
$\tx_n(0)=0, \ty_n(0)= 0, \tz_n(0)= 0$ for all $n\ge 0$ and $\tw_n(0)=0$ for all $n\ge 1$.
But for $\tw_0(v)$ we have
from \eqref{equ:DEsystem}
\begin{equation*}
w_0(0)=1, \quad w_d(0)=\frac{2d-1}{2d}w_{d-1}(0) \quad \forall d\ge 1.
\end{equation*}
It follows that $w_d(0)=\binom{2d}{d}/2^{2d}$ which yields easily
\begin{equation*}
\tw_0(v)=\sum_{d\ge 0}  w_d(0) v^d=(1-v)^{-\frac{1}{2}}.
\end{equation*}
Similarly we see that $\tz_0(v)=0$. Solving \eqref{equ:DEsystemTilde} recursively starting
from the first two equations in \eqref{equ:DEsystemTilde}
we find the following functions are the unique solution satisfying the initial conditions:
\begin{equation*}
\left\{\aligned
\tx_n(v)&=\sum_{j=0}^{2n+1} \frac{2^n(-1)^{\lfloor\frac{n+2}{2}\rfloor+j}}{j!(2n+1-j)!} (1-v)^{\frac{j-2}{4}};\\
\ty_n(v)&=\sum_{j=0}^{2n+1} \frac{2^n(-1)^{\lfloor \frac{n+3}{2}\rfloor+j}}{j!(2n+1-j)!}(1-v)^{\frac{j-2}{4}};\\
\tz_n(v)&=(1-(-1)^n)\sum_{j=0}^{2n}\frac{2^{n-1}(-1)^{\frac{n-1}{2}+j}}{j!(2n-j)!}(1-v)^{\frac{j-2}{4}};\\
\tw_n(v)&=(1+(-1)^n)\sum_{j=0}^{2n}\frac{2^{n-1}(-1)^{\frac{n}{2}+j}}{j!(2n-j)!}(1-v)^{\frac{j-2}{4}}.\\
\endaligned\right.
\end{equation*}
Using \eqref{equ:genFall} we can solve $x_n(v), y_n(v), z_n(v)$ and $w_n(v)$ and get
\begin{align*}
x_d(u)&= \sum_{n=0}^{\lfloor\frac{d-1}{2}\rfloor}\sum_{j=0}^{2n+1}\frac{2^n(-1)^{\lfloor\frac{n+2}{2}\rfloor+j+d}}{(2n+1)!}
    \binom{2n+1}{j}\binom{\frac{j-2}{4}}{d}u^n;\\
y_d(u)&= \sum_{n=0}^{\lfloor\frac{d-1}{2}\rfloor}\sum_{j=0}^{2n+1}\frac{2^n(-1)^{\lfloor\frac{n+3}{2}\rfloor+j+d}}
    {(2n+1)!}\binom{2n+1}{j}\binom{\frac{j-2}{4}}{d} u^n;\\
z_d(u)&= \sum_{n=0}^{2\lfloor\frac{d-2}{4}\rfloor +1}\sum_{j=0}^{2n}(1-(-1)^n)
    \frac{2^{n-1}(-1)^{\frac{n-1}{2}+j+d}}{(2n)!}\binom{2n}{j}\binom{\frac{j-2}{4}}{d}u^n;\\
w_d(u)&= \sum_{n=0}^{2\lfloor \frac{d}{4}\rfloor} \sum_{j=0}^{2n}
    (1+(-1)^n)\frac{2^{n-1}(-1)^{\frac{n}{2}+j+d}}{(2n)!}\binom{2n}{j}
\binom{\frac{j-2}{4}}{d} u^n.
\end{align*}
Thus
\begin{align*}
X_d(s)&= \sum_{n=0}^{\lfloor\frac{d-1}{2}\rfloor}\sum_{j=0}^{2n+1}\frac{2^n(-1)^{\lfloor\frac{n+2}{2}\rfloor+j+d}}{(2n+1)!}
    \binom{2n+1}{j}\binom{\frac{j-2}{4}}{d}s^{\frac{n}{2}};\\
Y_d(s)&=\sum_{n=0}^{\lfloor\frac{d-1}{2}\rfloor}\sum_{j=0}^{2n+1}\frac{2^n(-1)^{\lfloor\frac{n+3}{2}\rfloor+j+d}}
    {(2n+1)!}\binom{2n+1}{j}\binom{\frac{j-2}{4}}{d} s^{\frac{n}{2}};\\
Z_d(s)&= \sum_{n=0}^{\lfloor\frac{d-2}{4}\rfloor}\sum_{j=0}^{4n+2}
    \frac{2^{2n+1}(-1)^{n+j+d}}{(4n+2)!}\binom{4n+2}{j}\binom{\frac{j-2}{4}}{d} s^{n+1/2};\\
W_d(s)&= \sum_{n=0}^{\lfloor \frac{d}{4}\rfloor} \sum_{j=0}^{4n}
    \frac{2^{2n}(-1)^{n+j+d}}{(4n)!}\binom{4n}{j}
\binom{\frac{j-2}{4}}{d} s^n.
\end{align*}
By the well-known formulas
\begin{equation*}
z\cot z =-2\sum_{n=0}^{\infty} \frac{\zeta(2n)}{{\pi}^{2n}} z^{2n}, \quad
z\coth z= -2\sum_{n=0}^{\infty} (-1)^n\frac{\zeta(2n)}{{\pi}^{2n}}z^{2n},
\end{equation*}
we obtain
\begin{equation*}
\sqrt[4]{s}\cot \sqrt[4]{s}= -2\sum_{n=0}^{\infty} \frac{\zeta(2n)}{{\pi}^{2n}} s^{\frac{n}{2}},  \quad
\sqrt[4]{s}\coth \sqrt[4]{s}=-2\sum_{n=0}^{\infty} (-1)^{n}\frac{\zeta(2n)}{{\pi}^{2n}}s^{\frac{n}{2}},
\end{equation*}
and
\begin{equation*}
 \sqrt{s} \cot \sqrt[4]{s}\cdot \coth \sqrt[4]{s} = 4\sum_{k=0}^{\infty}
 \sum_{m+l=k} (-1)^{m}\frac{\zeta(2m)\zeta(2l)}{\pi^{2k}}s^{\frac{k}{2}}
 = 4\sum_{k=0}^{\infty}
 \sum_{m+l=2k} (-1)^{m}\frac{\zeta(2m)\zeta(2l)}{\pi^{4k}}s^{k}.
\end{equation*}
Here by exchanging $m$ and $l$ we notice that the inner sum vanishes if $k$ is odd.
Hence the coefficient of $s^n$ in $G_d(\pi^4 s)$ is
\begin{align*}
Q(4n,d)=&2\sum_{k=0}^{\lfloor\frac{d-1}{2}\rfloor}
\sum_{j=0}^{2k+1}\frac{2^k(-1)^{\lfloor\frac{k}{2}\rfloor +j+d}}{(2k+1)!}
    \binom{2k+1}{j}\binom{\frac{j-2}{4}}{d} \zeta(4n-2k)\pi^{2k}\\
+&2\sum_{k=0}^{\lfloor\frac{d-1}{2}\rfloor}
    \sum_{j=0}^{2k+1}(-1)^{k}\frac{2^{k}(-1)^{\lfloor\frac{k+1}{2}\rfloor+j+d}}{(2k+1)!}
    \binom{2k+1}{j}\binom{\frac{j-2}{4}}{d} \zeta(4n-2k)\pi^{2k}\\
+&4\sum_{k=0}^{\lfloor\frac{d-2}{4}\rfloor}\sum_{j=0}^{4k+2}
    \frac{2^{2k+1}(-1)^{k+j+d}}{(4k+2)!}\binom{4k+2}{j}\binom{\frac{j-2}{4}}{d}
    \left(\sum_{\substack{m,l\ge 0, \\ m+l=2n-2k}}
    (-1)^{m}\zeta(2m)\zeta(2l)\right) \pi^{4k}
\end{align*}
since $W_d(s)$ has degree less than $n$. Observe that the first two lines are the same and
for any positive integer $w$
\begin{align*}
 \sum_{\substack{m,l\ge 0, \\ m+l=2w }}  (-1)^{m}\zeta(2m)\zeta(2l)
=&2\sum_{l=1}^{w-1}\zeta(4l)\zeta(4w-4l)-\sum_{l=1}^{2w-1}\zeta(2l)\zeta(4w-2l)-\zeta(4w)\\
=& 4Q(4w,2)+(2w-3)\zeta(4w)-\frac{4w+1}{2}\zeta(4w)\\
=& 4Q(4w,2)-\frac{7}{2}\zeta(4w)
\end{align*}
by stuffle relation $\zeta(4m)\zeta(4l)=\zeta(4m,4l)+\zeta(4l,4m)+\zeta(4m+4l)$
and equation \eqref{equ:dbld=2}. Therefore we finally get
\begin{multline*}
Q(4n,d)=4\sum_{k=0}^{\lfloor\frac{d-1}{2}\rfloor}
\sum_{j=0}^{2k+1}\frac{2^k(-1)^{\lfloor\frac{k}{2}\rfloor +j+d}\zeta(4n-2k)\pi^{2k}}{(2k+1)!} \binom{2k+1}{j}\binom{\frac{j-2}{4}}{d}\zeta(4n-2k)\pi^{2k} \\
+4 \sum_{k=0}^{\lfloor\frac{d-2}{4}\rfloor}\sum_{j=0}^{4k+2}
    \frac{2^{2k+1}(-1)^{k+j+d}}{(4k+2)!}\binom{4k+2}{j}\binom{\frac{j-2}{4}}{d}
    \left(4Q(4n-4k,2)-\frac{7}{2}\zeta(4n-4k) \right) \pi^{4k}.
\end{multline*}
This concludes the proof of Theorem \ref{thm:main} and this paper.

\end{document}